\documentclass[11pt]{article}

\usepackage[a4paper,margin=1.15in]{geometry}

\usepackage{amssymb,amsmath, mathtools, amsthm,latexsym,booktabs,units,todonotes,color,enumerate,comment,lineno,graphics}

\usepackage{float}
\usepackage{hyperref}
\usepackage{cleveref}
\usepackage{microtype}
\usepackage{tikz}
\usetikzlibrary{calc}

\hypersetup{
  colorlinks=true,
  linkcolor=purple,
  citecolor=black,
  urlcolor=blue
}

\usepackage{url}
\usepackage[
    backend=biber,
    style=numeric,
    sorting=nyt
]{biblatex}
\AtBeginDocument{%
  \setlength{\abovedisplayskip}{8pt plus 2pt minus 4pt}%
  \setlength{\belowdisplayskip}{8pt plus 2pt minus 4pt}%
  \setlength{\abovedisplayshortskip}{6pt plus 2pt minus 3pt}%
  \setlength{\belowdisplayshortskip}{6pt plus 2pt minus 3pt}%
}

\theoremstyle{plain}
\newtheorem{thm}{Theorem}[section]
\newtheorem{prop}[thm]{Proposition}
\newtheorem{lem}[thm]{Lemma}
\newtheorem{obs}[thm]{Observation}
\newtheorem{cor}[thm]{Corollary}

\newtheorem{prob}{Problem}

\theoremstyle{definition}
\newtheorem{defn}[thm]{Definition}
\newtheorem{example}[thm]{Example}

\theoremstyle{remark} 
\newtheorem{rmk}[thm]{Remark} 
\newtheorem{claim}{Claim} 
\AtBeginEnvironment{thm}{\setcounter{claim}{0}}
\newtheorem{fact}{Fact} 
\newtheorem*{note}{Note} 
\newtheorem*{notation}{Notation}

\usepackage{setspace}
\newcommand{\HC}{\textnormal{HC}}
\newcommand{\C}{\textnormal{C}}
\newcommand{\MHC}{\textnormal{MHC}}
\newcommand{\MC}{\textnormal{MC}}
\newcommand{\mNHC}{\textnormal{mNHC}}
\newcommand{\mNC}{\textnormal{mNC}}

\title
{Sufficiency of Hall's Condition for Graphic List Coloring}

\author{
Parikshit Chalise\thanks{ Corresponding author.}
}
\date{\today}

\begin{document}


\maketitle


\begin{abstract}
For finite simple graphs $G,H$ on a common vertex set $V$, we say that $H$ is $G$-colorable if $H$ admits a proper list coloring with list assignment $L(v)=N_G(v)$ for all $v\in V$. This notion of coloring a graph using the neighborhood of another graph on the same vertex set, which we call \textit{graphic list coloring}, has connections to several classical topics, including systems of distinct representatives and graph factorizations. In this paper, we investigate when a necessary Hall-type condition, introduced by Hilton and Johnson in 1990, is also sufficient for $H$ to be $G$-colorable. We characterize all graphs $H$ that are $G$-colorable whenever the pair $(H,G)$ satisfies Hall's condition, answering a question raised by Johnson. We then consider the dual problem of characterizing graphs $G$ such that, whenever $(H,G)$ satisfies Hall's condition, $H$ is $G$-colorable. In this vein, we obtain complete results for several families of graphs, such as forests, complete multipartite graphs, and grid graphs.
\end{abstract} 



\section{Introduction}\label{sec:intro}
All graphs here are finite and simple. For a graph $H$, we write $V(H)$ for its vertex set and $E(H)$ for its edge set. Given a set $S$ of colors, let $L:V(H)\to 2^S$ be a list assignment to the vertices of $H$ with subsets of $S$. The graph $H$ is said to be \textit{$L$-colorable} if there exists a map $\varphi:V(H)\to S$ such that 
\begin{enumerate}[(i)]
    \item $\varphi(v)\in L(v)$ for all $v\in V(H)$, and 
    \item $\varphi(v)\neq \varphi(u)$ for all $uv\in E(H)$. 
\end{enumerate}
The second condition above is equivalent to saying $\varphi^{-1}(\sigma)$ is an independent set for all $\sigma\in S$. This notion of properly coloring a graph subject to specified constraints at each vertex is known as \textit{list coloring} and has a rich literature in graph theory  \cite{Alon1993RestrictedColorings, ErdosRubinTaylor1979Choosability, vizing}.

Given $\sigma\in S$, a \emph{$\sigma$-transversal} in $H$ is an independent set $X\subseteq V(H)$ such that $\sigma\in L(v)$ for all $v\in X$. Let $\alpha(H,L,\sigma)$ denote the maximum size of a $\sigma$-transversal in $H$. Throughout the paper, we say that $(H,L)$ satisfies \emph{Hall's condition} if for every subgraph $F$ of $H$,
\begin{equation}\label{eq:Hall}\tag{$\star$}
  |V(F)| \;\le\; \sum_{\sigma\in S}\alpha(F,L,\sigma).
\end{equation}
This definition is motivated by Hall's marriage theorem \cite{Hall1935}, which can be stated as follows: the complete graph $K$ is $L$-colorable if and only if $(K,L)$ satisfies \eqref{eq:Hall}. Hilton and Johnson \cite{Hilton-Johnson} discovered this formulation and showed that if a graph $H$ is $L$-colorable, then $(H,L)$ satisfies Hall's condition. The converse is not true in general, as we will discuss shortly.

In the current inquiry, we consider a special setting of list coloring, which we call \textit{graphic list coloring}. Let us denote by $\mathcal{G}(V)$ the set of all finite, simple, undirected graphs on the vertex set $V$. A list $L$ indexed by $V$ is \textit{graphic} if there exists a graph $G\in\mathcal G(V)$ such that
\[
L(v)=N_G(v)=\{u\in V:uv\in E(G)\}
\quad\text{for all }v\in V.
\]
This terminology is analogous to that of a graphic degree sequence.
Given $G,H \in \mathcal{G}(V)$, let $L$ be the graphic list defined by $L(v)=N_G(v)$ for all $v\in V$. We say that $H$ is $G$-colorable if $H$ is $L$-colorable. Likewise, we say that $(H,G)$ satisfies Hall's condition if $(H,L)$ satisfies Hall's condition.
 
The current paper concerns the following problems, posed by Peter D. Johnson during the
{Virtual Masamu Advanced Study Institute} (MASI) 2025.

\begin{prob}\label{prob:A} Characterize the graphs $H\in\mathcal G(V)$ such that, for every $G\in\mathcal G(V)$, if $(H,G)$ satisfies Hall's condition, then $H$ is $G$-colorable. \end{prob} 

\begin{prob}\label{prob:B} Characterize the graphs $G\in\mathcal G(V)$ such that, for every $H\in\mathcal G(V)$, if $(H,G)$ satisfies Hall's condition, then $H$ is $G$-colorable. \end{prob}

\subsection{Organization of  the paper}
In Section~\ref{sec:background}, we review relevant previous work. In Section~\ref{sec:prob1}, we obtain a complete answer for Problem~\ref{prob:A}. For Problem~\ref{prob:B}, we obtain results for several families of graphs, including forests,  complete multipartite graphs, and grid graphs; these results are contained in Section~\ref{sec:prob2} of the paper. We conclude with some future directions in Section~\ref{sec:conclusion}.

\section{Preliminaries}\label{sec:background}

\subsection{Graphic list coloring}
A systematic study of graphic list coloring was initiated recently in \cite{self-coloring}. One of the original motivations was to characterize graphs whose open-neighborhood lists admit a system of distinct representatives (SDR) \cite{hedet2018}. Equivalently, this asks for a characterization of the graphs $G$ for which the complete graph on the same vertex set is $G$-colorable. An answer to this was obtained in terms of graph factors.

\begin{defn}
   Given a graph $H$, a subgraph $F \subseteq H$ is a \textit{spanning $\{1,2\}$-factor} of $H$ if $V(F)=V(H)$ and each connected component of $F$ is either $1$-regular or $2$-regular. In other words, $F$ is a collection of single edges and cycles. 
\end{defn} 

\begin{thm}[\cite{self-coloring, hedet2018}]\label{thm:hedet-main}
Given a graph $G$, the complete graph on $V(G)$ is $G$-colorable if and only if $G$ has a spanning $\{1,2\}$-factor.
\end{thm}

The following monotonicity facts on graphic list coloring follow from the definitions and are derived from \cite{self-coloring}.
\begin{prop}[\cite{self-coloring}]\label{prop:basics} 
    Suppose $H,G,X\in \mathcal{G}(V)$. 
\begin{enumerate}
    \item If $G$ has an isolated vertex, then no graph on $V$ is $G$-colorable.
    \item If $H$ is $G$-colorable and $X$ is a spanning subgraph of $H$, then $X$ is $G$-colorable.
\item   If $H$ is $G$-colorable and $G$ is a spanning subgraph of $X$, then $H$ is $X$-colorable.
\end{enumerate}
\end{prop}

The main implication of  Proposition~\ref{prop:basics} is that if a graph $G$ on the vertex set $V$ colors the complete graph on $V$, then it colors every graph on $V$.
The following result was also obtained by the authors of \cite{self-coloring}.

\begin{thm}[\cite{self-coloring}]\label{thm:self-coloring}
For every graph $G$ with no isolated vertices, $G$ is $G$-colorable.
\end{thm}

\subsection{Hall's condition}
Systems of distinct representatives (SDR) arise naturally in matching theory, most notably through Hall's marriage theorem \cite{Hall1935}. The following foundational fact discovered by Hilton and Johnson \cite{Hilton-Johnson} married SDR with graph coloring.

\begin{lem}[\cite{Hilton-Johnson}]\label{lem:HJ-lem}
If $H$ is $L$-colorable, then $(H,L)$ satisfies Hall's condition.
\end{lem}
As noted earlier, Hall's condition is not sufficient for list colorability in general. For instance, Figure~\ref{fig:not-hall-L-sufficient} shows an example of a graph $H$ such that  $(H,L)$ satisfies Hall's condition, but $H$ is not $L$-colorable. Removing the edge $v_1v_3$ yields another  example.
This naturally leads to the question of determining those graphs $H$, besides the complete graph, for which Hall's condition is not only necessary but also sufficient for $L$-colorability. Let us call such a graph \textit{Hall-colorable}  to indicate sufficiency of Hall's condition for colorability. The characterization of Hall-colorable graphs is complete.

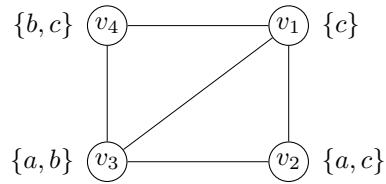
\begin{figure}[H]
\centering
\begin{tikzpicture}[
  x=1.2cm, y=0.9cm,
  v/.style={circle,draw,inner sep=1.4pt},
  col/.style={circle,draw,inner sep=1.4pt},
  blk/.style={draw,rounded corners,inner sep=4pt,align=center},
  every node/.style={font=\small},
  edge/.style={line width=0.35pt}
]

\node[v] (v0) at (0,  0) {$v_1$};
\node[v] (v1) at (0,  -2) {$v_2$};
\node[v] (v3) at (-2,  0) {$v_4$};
\node[v] (v2) at (-2,  -2) {$v_3$};

\node[right=1pt of v0] {$\{c\}$};
\node[left=1pt of v3] {$\{b,c\}$};
\node[right=1pt of v1] {$\{a,c\}$};
\node[left=1pt of v2] {$\{a,b\}$};

\draw[edge] (v0) -- (v1);
\draw[edge] (v1) -- (v2);
\draw[edge] (v2) -- (v3);
\draw[edge] (v0) -- (v3);
\draw[edge] (v2) -- (v0);
\end{tikzpicture}
\caption{A graph $H$ such that $(H,L)$ satisfies Hall's condition, but $H$ is not $L$-colorable.}
\label{fig:not-hall-L-sufficient}  
\end{figure} 

\begin{thm}[\cite{Hilton-Johnson}]\label{thm:HJ-thm}
A graph $H$ is Hall-colorable if and only if it contains neither an induced cycle $C_k$, $k\ge 4$, nor an induced copy of the diamond graph $K_4-e$.
\end{thm}

In other words, the family of graphs for which Hall's condition is sufficient for list colorability consists of those graphs $H$ such that every \textit{block} in $H$, i.e., a maximal connected subgraph with no cut vertex, is a clique. We remind the reader that Theorem~\ref{thm:HJ-thm} does not completely answer Problem~\ref{prob:A}, for the lists in the current study are not arbitrary;  they must be graphic.

Subsequent work on Hall's condition studied quantitative parameters similar to the choice number for list coloring, measuring how far a graph is from having Hall's condition suffice for list colorability, including the Hall number, Hall index, and total Hall number \cite{CropperHilton2002HallParameters, HiltonJohnson1999HallNumber, Johnson2002HallParameters}. There have also been some results employing Hall's condition to the problem of completing partial proper colorings \cite{BobgaGoldwasserHiltonJohnson2011CompletingPartialLatinSquares, HollidayVandenbusscheWestlund2015CompletingPartialProperColorings}.

In general, checking Hall's condition is a laborious process as it needs to be verified for every connected induced subgraph,\footnote{This is equivalent to satisfying Hall's condition \eqref{eq:Hall} for every subgraph.} potentially on the order of $2^{|V|}$. The following two observations, which can be found in \cite{Cropper1998HallsConditionListColoring}, simplify the task to some extent. The first observation below is a consequence of Lemma~\ref{lem:HJ-lem} and the second one directly follows from the definition of Hall's condition.

\begin{obs}\label{obs:hall-check}
If $H-v$ is $L$-colorable for all $v \in V(H)$,  then $(H,L)$ satisfies Hall's condition if and only if $  |V(H)| \;\le\; \sum_{\sigma\in S}\alpha(H,L,\sigma).
$
\end{obs}

\begin{obs}\label{obs:hall-partition}
If there exist partitions $S_1\sqcup \cdots \sqcup S_k$ of the set $S$ and $V_1\sqcup \cdots \sqcup V_k$  of the vertex set $V(H)$ such that the subgraph of $H$ induced by $V_i$ and the restricted list assignment $L(v) \cap S_i$ satisfy Hall's condition for all $i=1,\ldots,k$, then the pair $(H,L)$ satisfies Hall's condition.
\end{obs}

\begin{example}
In light of Observation~\ref{obs:hall-check}, one only needs to verify Hall's condition once, for $H$ itself. The reader is welcome to do so for the graph $H$ and the list assignment $L$ in Figure~\ref{fig:not-hall-L-sufficient}, as it is readily seen that $H-v$ is $L$-colorable for every $v\in V(H)$.
\end{example}

When lists are graphic as in the current context, we obtain the following facts analogous to results in Proposition~\ref{prop:basics}. Proofs are omitted, as these follow from the definition of Hall's condition. 
\begin{prop}\label{prop:basics-Hall}
    Suppose $H,G,X\in \mathcal{G}(V)$. 
\begin{enumerate}
\item If $G$ has an isolated vertex, then no pair $(H,G)$ satisfies Hall's condition.
    \item   If $(H,G)$ satisfies Hall's condition and $X$ is a spanning subgraph of $H$, then $(X,G)$ satisfies Hall's condition.
    \item If $(H,G)$ satisfies Hall's condition and $G$ is a spanning subgraph of $X$, then $(H,X)$ satisfies Hall's condition.
\end{enumerate}
\end{prop}

As an immediate consequence of Theorem~\ref{thm:self-coloring} and Lemma~\ref{lem:HJ-lem}, we also obtain the following:

\begin{cor}\label{cor:self-coloring-cor}
For every graph $G$ with no isolated vertices, the pair $(G,G)$ satisfies Hall's condition.
\end{cor}

\section{$\mathcal{G}$-Hall-colorable Graphs}\label{sec:prob1}

The current section focuses on Problem~\ref{prob:A}.  We refer to the graphs satisfying the property in Problem~\ref{prob:A} as {$\mathcal G$-Hall-colorable}, as defined below.

\begin{defn}\label{def:G-Hall-colorable}
A graph $H\in\mathcal G(V)$ is called \textit{$\mathcal G$-Hall-colorable} if, for every $G\in\mathcal G(V)$, whenever $(H,G)$ satisfies Hall's condition, $H$ is $G$-colorable.
\end{defn}

\begin{note}
Since the term Hall-colorable was used in Theorem~\ref{thm:HJ-thm} for arbitrary list assignments, we introduce the prefix $\mathcal G$ in Definition~\ref{def:G-Hall-colorable} to emphasize that the list is graphic.
\end{note}

Let $\mathcal{F}$ denote the forbidden set of induced subgraphs in  Hilton-Johnson's characterization of sufficiency of Hall's condition via Theorem~\ref{thm:HJ-thm}. 
Explicitly,
\begin{equation}\label{eq:forbidden}
     \mathcal{F} = \{C_k:k\geq 4 \} \cup \{K_4-e\}.
\end{equation}
Hilton-Johnson's characterization of Hall-colorable graphs relied on their finding that graphs in $\mathcal{F}$ can be assigned a fatal list $L$, which is capable of satisfying Hall's condition, yet being unable to color them. The following argument already shows that not all counterexamples in $\mathcal{F}$ remain when the list $L$ is asked to be graphic.
\begin{prop}\label{prop:diamond}
        The diamond graph $K_4-e$ is $\mathcal{G}$-Hall-colorable.
\end{prop}
\begin{proof}
 We have to show $K_4-e$ is $G$-colorable whenever $(K_4-e,G)$ satisfies Hall's condition, where both graphs are on the same vertex set. By Proposition~\ref{prop:basics-Hall}, $G$ cannot have an isolated vertex, for otherwise $(K_4-e,G)$ does not satisfy Hall's condition. Let $G$ be a graph on 4 vertices with no isolated vertices. If $G$ contains a perfect matching, then $G$ colors $K_4$ by swapping colors of the adjacent vertices in a perfect matching. Hence, $G$ colors the spanning subgraph $K_4-e$ by Proposition~\ref{prop:basics}. The only remaining case is $G\cong K_{1,3}$. However, in this case, $(K_4-e,K_{1,3})$ does not satisfy Hall's condition. Indeed, $K_{1,3}$ supplies three identical singleton lists to three vertices of $K_4-e$, and they must form an independent set of size 3 in $K_4-e$ to satisfy Hall's condition. But the independence number of $K_4-e$ is 2. This completes the proof.
 \end{proof}
 
 Note that the argument in the proof of Proposition~\ref{prop:diamond} also shows $C_4$ is $\mathcal G$-Hall-colorable. This means that the classical obstructions to Hall-colorability are no longer enough when the lists are graphic. We now proceed to provide a complete answer to Problem~\ref{prob:A}.

\subsection{Technical lemmas}

\begin{lem}\label{lem:2-color-Hall}
Let $H \in \mathcal{G}(V)$, and suppose $S$ is a set of colors such that $|S|\leq 2$. Let $L:V \to 2^S$ be some list assignment. If $(H,L)$ satisfies Hall's condition, then $H$ is $L$-colorable.
\end{lem}
\begin{proof}
If $|S| = 1$, then $H$ must be an edgeless graph where each vertex list is a singleton, in which case the result holds. So suppose $S = \{a,b\}$. First, if $(H,L)$ satisfies Hall's condition, then $H$ is bipartite. Indeed, if $H$ contains an odd cycle $C$ of order $r$, then for each  $\sigma \in \{a,b\}$, we have $\alpha(C,L,\sigma) \leq \alpha(C) = \frac{r-1}{2}$. Consequently, we have
\[
\sum_{\sigma \in \{a,b\}} \alpha(C,L,\sigma) \leq 2\cdot\frac{r-1}{2} < r,
\]
contrary to Hall's condition.

We may assume,  without loss of generality, $H$ has a single connected component. Let $V=A\sqcup B$ be the bipartition of $H$. We prove the contrapositive. Suppose $H$ is not $L$-colorable. Then at least two vertices $u,v \in V$ receive a singleton list such that either both $u,v$ lie in the same partition and $L(u) \neq L(v)$ or vertices $u,v$ lie in distinct partitions and $L(v) = L(u)$. Pick such a pair at minimum distance and let $P$ be the $u–v$ path of order $r$. Every internal vertex of $P$ must have the list $\{a,b\}$, for otherwise there would be a shorter path. If $L(v) \neq L(u)$, then $P$ is of odd order. In this case, 
\[
\alpha(P,L,a)=\alpha(P,L,b)=\frac{r-1}{2},
\]
which implies
\[
\sum_{\sigma\in\{a,b\}}\alpha(P,L,\sigma)=r-1.
\]
Next suppose, without loss of generality, $L(u)= L(v) = \{a\}$. Then $P$ is of even order. In this case,
\[
\alpha(P,L,a)=\frac r2
\qquad\text{and}\qquad
\alpha(P,L,b)=\frac{r-2}{2},
\]
and again
\[
\sum_{\sigma\in\{a,b\}}\alpha(P,L,\sigma)=r-1.
\]
Both cases contradict Hall's condition, as required. 
\end{proof}

As we have remarked earlier, the characterization of Hall-colorable graphs in Theorem~\ref{thm:HJ-thm} relied on carefully chosen list assignments for graphs in $\mathcal F$ so that they satisfy Hall's condition but obstruct proper coloring.  The technical lemma below constructs this obstruction for cycles by using a total of 3 colors. This construction is useful for a later result.

\begin{lem}\label{lem:cycle-Hall-not-colorable}
Let $k\ge 4$ and let $C_k$ be the cycle 
$v_1–v_2–\cdots–v_{k}–v_1$.
Let $S=\{a,b,c\}$ and define a list assignment $L:V(C_k)\to 2^S$ by
\[
L(v_1)=\{c\},\qquad L(v_2)=\{a,c\},\qquad L(v_i)=\{a,b\}\,\,\text{ for all }3\le i\le k-1,
\]
and
\[
L(v_{k})=
\begin{cases}
\{b,c\}, & \text{if $k$ is even},\\[2mm]
\{a,c\}, & \text{if $k$ is odd}.
\end{cases}
\]
Then $(C_k,L)$ satisfies Hall's condition, but $C_k$ is not $L$-colorable.
\end{lem}
\begin{proof}
First, we show that $(C_k,L)$ satisfies Hall's condition. We proceed in the manner informed by Observation~\ref{obs:hall-check}. One can easily verify that for every vertex $v\in V(C_k)$ the graph $C_k-v$ is $L$-colorable.
Then by Lemma~\ref{lem:HJ-lem}, every connected induced subgraph of $C_k-v$ satisfies Hall's
condition. We now check that Hall's condition is satisfied for $C_k$ itself. We compute $\alpha(C_k,L,\sigma)$ for $\sigma\in\{a,b,c\}$. For all $k \geq 4$, $\alpha(C_k,L,c)= 2$. 
If $k$ is even, $\alpha(C_k,L,a)=\frac{k-2}{2}$ and $\alpha(C_k,L,b)=\frac{k-2}{2}$. 
If $k$ is odd,
$\alpha(C_k,L,a)=\frac{k-1}{2}$ and $\alpha(C_k,L,b)=\frac{k-3}{2}$. In either case, 
 \[
  |V(C_k)|=k=\sum_{\sigma\in S}\alpha(C_k,L,\sigma),
 \]
meaning $(C_k,L)$ satisfies Hall's condition.
 
Next, we show $C_k$ is not $L$-colorable. Suppose for contradiction that $\varphi$ is a proper $L$-coloring of $C_k$. We must have $\varphi(v_1)=c$ and $\varphi(v_2)=a$.
For each $i=3,\dots,k-1$, $\varphi(v_i)$ alternates assigning $b$ or $a$.
Consequently,
\[
\varphi(v_{k-1})=
\begin{cases}
b, & \text{if $k$ is even},\\
a, & \text{if $k$ is odd}.
\end{cases}
\]
By construction, $L(v_{k}) = \{ \varphi(v_{k-1}), \varphi(v_1)\}$. But $v_{k}$ is adjacent to both $v_{k-1}$ and $v_1$, which yields a contradiction.
\end{proof}

\subsection{Key lemma}

\begin{notation}
Given $H\in \mathcal{G}(V)$ and $X\subseteq V$, we denote by $H[X]$ the induced subgraph of $H$ on the vertex set $X$. The set $N_H(X)$ is defined by $\bigcup_{x \in X}N_H(x)$. We use the symbol $\sqcup$ to denote a disjoint union.
\end{notation}

\begin{lem}\label{lem:no-12-factor-decomposition}
For every graph $G \in \mathcal{G}(V)$, there is a partition $V=A\sqcup B\sqcup R
$ such that
\begin{enumerate}
    \item $A$ is independent;
    \item $B=N_G(A)$, and there is a matching from $B$ into $A$ which saturates $B$;
    \item $G[R]$ has a spanning $\{1,2\}$-factor.
\end{enumerate}
\end{lem}

The proof of Lemma~\ref{lem:no-12-factor-decomposition} relies on the following characterization from \cite{self-coloring}, which can also be derived from a classical result of Tutte \cite{Tutte1953} or simply from Hall's marriage theorem \cite{Hall1935}.
\begin{fact}\label{fact:12-factor}
A graph $G$ has a spanning $\{1,2\}$-factor if and only if $|N_G(I)|\geq |I|$ for every independent set $I$ in $G$.
\end{fact}

\begin{proof}[Proof of Lemma~\ref{lem:no-12-factor-decomposition}]
If $G$ has a spanning $\{1,2\}$-factor, then we let $A=B=\varnothing$ and $R=V$. Hence suppose $G$ does not contain a spanning $\{1,2\}$-factor. By Fact~\ref{fact:12-factor}, there exists an independent set $I$ such that $|I| > |N_G(I)|$.  Define
\[
\eta(G)
=
\max\bigl\{
|I|-|N_G(I)|:
I\subseteq V\text{ is independent}
\bigr\},
\]
and choose an independent set $A$ such that $|A| - |N_G(A)|=\eta(G) >0$. Define the sets
\[
B = N_G(A) \qquad \text{and} \qquad R = V(G) \setminus (A \cup B).
\]

We first show that there is a matching from $B$ into $A$ which saturates $B$. For $U \subseteq B$, let $N_A(U) = N_G(U) \cap A$. By Hall's marriage theorem, it suffices to show $|N_A(U)| \geq |U|$. Suppose, to the contrary, there exists $U \subseteq B$ such that  $|N_A(U)| < |U|$. Note that the set $A'= A \setminus N_A(U)$
is independent and $N_G(A') \subseteq B \setminus U$. It follows that
\[
\begin{aligned}
|A'|-|N_G(A')|
&\geq
|A|-|N_A(U)|-|B|+|U|  \\
&=
\eta(G)+|U|-|N_A(U)|  \\
&>\eta(G),
\end{aligned}
\]
which contradicts the maximality of $\eta(G)$. Therefore, we conclude that there is a matching from $B$ into $A$ which saturates $B$. 

Next, we show $G[R]$ has a spanning $\{1,2\}$-factor. Again, for the sake of contradiction, suppose not. Then there is an independent set $W\subseteq R$ such that $|N_{G[R]}(W)|<|W|$.
By definition, there are no edges between $A$ and $R$, so $A \cup W$ is independent. Note that $N_G(A\cup W)\subseteq B\cup N_{G[R]}(W)$.
Then we have
\[
\begin{aligned}
|A\cup W|-|N_G(A\cup W)|
&\geq
|A|+|W|-|B|-|N_{G[R]}(W)|  \\
&>
|A|-|B|  \\
&=\eta(G),
\end{aligned}
\]
which again contradicts the maximality of $\eta(G)$.
We conclude $G[R]$ has a spanning
$\{1,2\}$-factor.
\end{proof}

\subsection{Sufficient conditions}
Recall the forbidden graph set $\mathcal F=\{C_k:k\geq 4\}\cup\{K_4-e\}$ as defined in \eqref{eq:forbidden}.

\begin{thm}\label{thm:G-Hall-characterization}
Suppose that for every induced subgraph $F \subseteq H$, where $F \in \mathcal{F}$, we have $|V(H)\setminus V(F)|\leq 2$. Then $H$ is $\mathcal G$-Hall-colorable.
\end{thm}
\begin{proof}
    Let $G$ be a graph
on $V(H)$ such that $(H,G)$ satisfies Hall's condition. We prove that
$H$ is $G$-colorable. By Proposition~\ref{prop:basics-Hall}, $G$ has no isolated vertex. By Theorem~\ref{thm:hedet-main}, we may assume that $G$ has no spanning $\{1,2\}$-factor, for otherwise $G$ colors every graph on $V(G)$, and there is nothing to prove. By Lemma~\ref{lem:no-12-factor-decomposition}, we may write $V(G)=A\sqcup B\sqcup R$, where $A$ is independent in $G$, $B=N_G(A)$ such that there is a matching from
$B$ into $A$ which saturates $B$, and $G[R]$ has a spanning
$\{1,2\}$-factor.

Note that we have $N_G(v) \subseteq B$ for all $v \in A$. We first show $H[A]$ can be properly colored. If $H[A]$ contains no induced member of $\mathcal F$, this follows from Theorem~\ref{thm:HJ-thm}.
Otherwise, let $F\in\mathcal F$ be induced in $H[A]$. The graph $F$ is also induced
in $H$, and $B\subseteq V(H)\setminus V(F)$. Hence, we have
\[
|B|\leq |V(H)\setminus V(F)|\leq 2.
\]
Then we must have $|N_G(v)| \leq 2$ for all $v\in A$.  By Lemma~\ref{lem:2-color-Hall}, we obtain that $H[A]$ admits a proper coloring using colors in $B$. Let $\psi : A \to B$ be this proper coloring. Since $G[R]$ contains a spanning $\{1,2\}$-factor, any graph $H[R]$ admits a proper coloring using colors in $R$, by Theorem~\ref{thm:hedet-main}. Let $\pi : R \to R$ be this proper coloring. Finally, let $\mu: B \to A$ be such that $\{(b,\mu(b)) :b\in B\}$ is a matching from $B$ into $A$ which saturates $B$. Then $\mu$ is a proper coloring for any graph $H[B]$ using colors in $A$. To summarize, we define $\varphi:V(H)\to V(H)$ by
\[
\varphi(v)=
\begin{cases}
\psi(v),& \text{if } v\in A,\\
\mu(v),&\text{if }v\in B,\\
\pi(v),&\text{if }v\in R.
\end{cases}
\]
By construction, we have $\varphi(v)\in N_G(v)$ for all $v\in V(H)$, and $\psi, \mu, \pi$ are proper colorings with pairwise disjoint codomains. Therefore, $\varphi$ is a proper $G$-coloring of $H$.
\end{proof}

\subsection{Necessary conditions}

\begin{thm}\label{thm:not-Hall}
Suppose that there exists an induced subgraph $F \subseteq H$, where $F \in \mathcal{F}$, such that $|V(H)\setminus V(F)|\geq 3$ and $|V(H)\setminus V(F)| \neq 4$. Then $H$ is not $\mathcal G$-Hall-colorable.
\end{thm}
\begin{proof}
It suffices to show that there exists some graph $G$ on $V(H)$ such that $(H,G)$ satisfies Hall's condition, but $H$ is not $G$-colorable.
 The proof is constructive and relies on Lemma~\ref{lem:cycle-Hall-not-colorable}. We show the construction via Figure~\ref{fig:not-Hall-universal}. 

Let $F$ be an induced subgraph of $H$ such that $F$ is either a cycle $C_k, \,k\ge 4$, or the diamond graph $K_4-e$. Since $|V(H)\setminus V(F)|\geq 3$, we may write $ V(H) = V(F) \sqcup \{a,b,c\} \sqcup R$, where $R$ contains the remaining vertices.
If $F = C_k$, where $C_k$ denotes the induced cycle $v_1–v_2–\cdots–v_{k}–v_1, k\geq 4$,  we construct a graph $G$ such that
\[
L(v_1)=\{c\},\qquad L(v_2)=\{a,c\},\qquad L(v_i)=\{a,b\}\,\,\text{ for all }3\le i\le k-1,
\]
and
\[
L(v_{k})=
\begin{cases}
\{b,c\}, & \text{if $k$ is even},\\[2mm]
\{a,c\}, & \text{if $k$ is odd}.
\end{cases}
\]
This is exactly the list assignment as described in Lemma~\ref{lem:cycle-Hall-not-colorable} and is implemented by the graph $G$ in Figure~\ref{fig:not-Hall-universal}, depending on whether the length of the induced cycle is odd or even. In Figure~\ref{fig:not-Hall-universal}, the remaining vertices in $R$ form a clique $K_R$.  If $F\cong K_4-e$, label it as in Figure~\ref{fig:not-hall-L-sufficient}, and the same list is supplied by graph $G$ in Figure~\ref{fig:not-Hall-universal} corresponding to even $k$. In both cases, $F$ satisfies Hall's condition with the given list assignment $L$, but $F$ is not $L$-colorable.
\begin{figure}[h]
\centering

\begin{minipage}{0.48\linewidth}
\centering
\begin{tikzpicture}[
  x=1.05cm, y=0.9cm,
  v/.style={circle,draw,inner sep=1.4pt},
  col/.style={circle,draw,inner sep=1.4pt},
  edge/.style={line width=0.35pt},
  cont/.style={line width=0.35pt,densely dotted},
  every node/.style={font=\small}
]

\node[v] (v1) at (0,2.8) {$v_1$};
\node[v] (v2) at (0,1.6) {$v_2$};
\node[v] (v3) at (0,0.4) {$v_3$};
\node[draw=none] (vd) at (0,-0.45) {$\vdots$};
\node[v] (vk1) at (0,-1.5) {$v_{k-1}$};

\node[col] (c) at (2.0,2.6) {$c$};
\node[col] (a) at (2.0,1.0) {$a$};
\node[col] (b) at (2.0,-0.6) {$b$};

\node[v] (vk) at (4.0,2.6) {$v_k$};
\node[v] (KR) at (4.0,0.4) {$K_R$};

\draw[edge] (v1) -- (c);
\draw[edge] (v2) -- (a);
\draw[edge] (v2) -- (c);

\draw[edge] (a) -- (v3);
\draw[cont] (a) -- (vd);
\draw[edge] (a) -- (vk1);

\draw[edge] (b) -- (v3);
\draw[cont] (b) -- (vd);
\draw[edge] (b) -- (vk1);


\draw[edge] (vk) -- (b);
\draw[edge] (vk) -- (c);

\node[draw=none] at (2.0,-2) {Even $k$};

\end{tikzpicture}
\end{minipage}
\hfill
\begin{minipage}{0.48\linewidth}
\centering
\begin{tikzpicture}[
  x=1.05cm, y=0.9cm,
  v/.style={circle,draw,inner sep=1.4pt},
  col/.style={circle,draw,inner sep=1.4pt},
  edge/.style={line width=0.35pt},
  cont/.style={line width=0.35pt,densely dotted},
  every node/.style={font=\small}
]

\node[v] (v1) at (0,2.8) {$v_1$};
\node[v] (v2) at (0,1.6) {$v_2$};
\node[v] (v3) at (0,0.4) {$v_3$};
\node[draw=none] (vd) at (0,-0.45) {$\vdots$};
\node[v] (vk1) at (0,-1.5) {$v_{k-1}$};

\node[col] (c) at (2.0,2.6) {$c$};
\node[col] (a) at (2.0,1.0) {$a$};
\node[col] (b) at (2.0,-0.6) {$b$};

\node[v] (vk) at (4.0,2.6) {$v_k$};
\node[v] (KR) at (4.0,0.4) {$K_R$};

\draw[edge] (v1) -- (c);
\draw[edge] (v2) -- (a);
\draw[edge] (v2) -- (c);

\draw[edge] (a) -- (v3);
\draw[cont] (a) -- (vd);
\draw[edge] (a) -- (vk1);

\draw[edge] (b) -- (v3);
\draw[cont] (b) -- (vd);
\draw[edge] (b) -- (vk1);


\draw[edge] (vk) -- (a);
\draw[edge] (vk) -- (c);

\node[draw=none] at (2.0,-2) {Odd $k$};

\end{tikzpicture}
\end{minipage}
\caption{A construction of graph $G$ which does not properly color the
cycle $v_1–v_2–\cdots–v_k–v_1$.}
\label{fig:not-Hall-universal}
\end{figure}
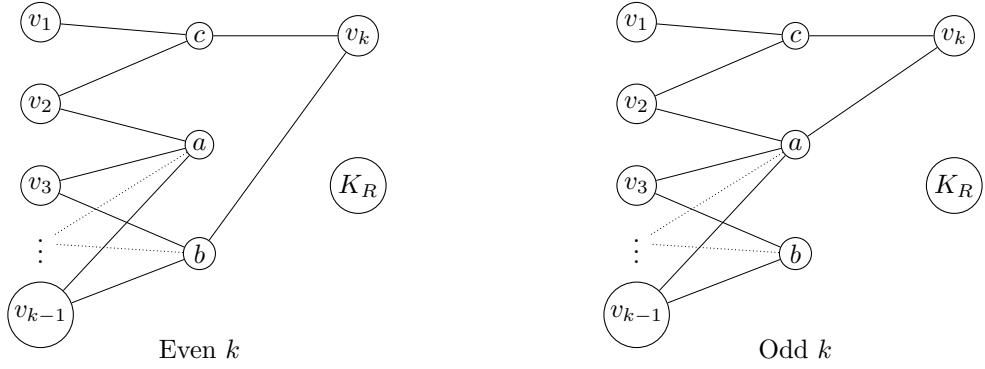

We now show $(H,G)$ satisfies Hall's condition. We utilize Observation~\ref{obs:hall-partition}.
It suffices to show that any induced subgraph on the vertex set $V(H)\setminus V(F)$ can be properly colored without using colors in $\{a,b,c\}$. In fact, we can color $V(H) \setminus V(F)$ injectively using $V(H) \setminus \{a,b,c\}$. The vertices $a,b,c$ can be properly colored by vertices in $V(F)$. For example, one can always define a proper coloring $\varphi$ such that $\varphi(c) = v_1, \varphi(a) = v_2$, and $\varphi(b) = v_3$. It remains to properly color $H[R]$, but this is $K_R$-colorable since $K_R$ has a $\{1,2\}$-factor and colors every graph $H[R]$. One exceptional case is when $R$ is a singleton; this is the case $|V(H) \setminus V(F)|=4$. This completes the proof.
\end{proof}

\subsection{Exceptional cases}\label{sec:exception}

One would hope to remove the small case $|V(H)\setminus V(F)| \neq 4$ from the statement of Theorem~\ref{thm:not-Hall}. However, when $|V(H)\setminus V(F)|=|V(F)|=4$, a surprising case emerges.

\begin{notation}
If $J$ is a subgraph of a graph $H$, the graph $H - J$ denotes the spanning subgraph of $H$ obtained by deleting the edges of $J$ from $H$.     For graphs $G_1$ and $G_2$, their join $G_1\vee G_2$ is obtained by adding all edges between vertices of $G_1$ and $G_2$.

\end{notation}

\begin{prop}\label{prop:K8-minus-Kst}
Let $s,t\ge1$ with $s+t\le6$. Then $K_8-K_{s,t}$
is $\mathcal G$-Hall-colorable. 
\end{prop}

In particular, Proposition~\ref{prop:K8-minus-Kst} yields several graphs with an induced diamond and 4 additional vertices, which are $\mathcal{G}$-Hall-colorable, showing that the excluded case cannot simply be incorporated into Theorem~\ref{thm:not-Hall}. 

\begin{proof}[Proof of Proposition~\ref{prop:K8-minus-Kst}]
Let $H \cong K_8-K_{s,t}$. Consider a graph $G$ on $V=V(H)$, and suppose that $(H,G)$ satisfies Hall's condition. If $G$ has a spanning
$\{1,2\}$-factor, there is nothing to prove. Otherwise, choose the partition $V=A\sqcup B\sqcup R$
as in the proof of Lemma~\ref{lem:no-12-factor-decomposition}, where $A$ is independent in $G$, $B=N_G(A)$, $|B| < |A|$, there is a matching from
$B$ into $A$ which saturates $B$, and $G[R]$ has a spanning
$\{1,2\}$-factor.

Since $H \cong K_8-K_{s,t}$, we have $\alpha(H)\le2$. As every color
available to a vertex of $A$ belongs to $B$, satisfying Hall's condition implies
\[
|A|
\le
\sum_{\sigma\in B}\alpha(H[A],L,\sigma)
\le 2|B|.
\]
Together with $|B|<|A|$, $|A|+|B|+|R|=8$, $|R|\neq1$, we obtain either
\[
|B|\le2 \qquad \text{or} \qquad |A|=5,\quad |B|=3,\quad R=\varnothing.
\]
If $|B|\le2$, then Lemma~\ref{lem:2-color-Hall} gives a proper
coloring of $H[A]$ using colors in $B$. It remains to consider $|A|=5, |B|=3$. For every clique $Q$ of $H[A]$, satisfying Hall's condition means
\[
|Q|
\le
\sum_{\sigma\in B}\alpha(Q,L,\sigma)
\le |B|=3,
\]
i.e., the clique number of $H[A]$ is at most 3. Since $H[A]$ is an induced subgraph of
$K_8-K_{s,t}$ on five vertices, the only possibilities are
\[
H[A]\cong K_3\sqcup K_2
\qquad\text{or}\qquad
H[A]\cong K_1\vee(K_2\sqcup K_2).
\]
In either case, every block of $H[A]$ is a clique, and Theorem~\ref{thm:HJ-thm} gives a proper coloring of $H[A]$ using colors in $B$.

Finally, we use the matching from $B$ into $A$ to color the vertices of $B$
injectively using colors in $A$. And the spanning
$\{1,2\}$-factor of $G[R]$ gives a proper coloring of $H[R]$ using
colors in $R$.  As in the proof of Theorem~\ref{thm:G-Hall-characterization}, the three color sets $A$, $B$, and $R$ are pairwise disjoint, so they collectively yield a proper $G$-coloring of $H$.
\end{proof}

However, it turns out Proposition~\ref{prop:K8-minus-Kst} characterizes all exceptions.

\begin{prop}\label{prop:four-extra-characterization}
Suppose that there exists an induced subgraph $F \subseteq H$, where $F \in \mathcal{F}$, such that $|V(H)\setminus V(F)|= 4$. If $H\not\cong K_8-K_{s,t}$
for all $s,t\ge1$ with $s+t\le6$, then $H$ is not
$\mathcal G$-Hall-colorable.
\end{prop}

\begin{proof}

First, if $F\cong C_k$,  $k\geq 5$, we construct a bipartite
graph $G$ with independent sets $V(F)=\{v_1,\ldots,v_k\}$ and $\{a,b,c,d\}$,
such that
\[
L(v_1)=\{c\},\quad
L(v_i)=\{a,c\}\,\,\text{ for all }2\le i\le k-3, \quad
L(v_{k-1})=\{b,d\},\quad
L(v_k)=\{c,d\},
\]
and
\[
L(v_{k-2})=
\begin{cases}
\{b,c\}, & \text{if $k$ is even},\\[2mm]
\{a,b\}, & \text{if $k$ is odd}.
\end{cases}
\]
The said list assignment is such that $(H,L)$ satisfies Hall's condition, but the subgraph $F$ is not $L$-colorable, meaning $H$ is not $G$-colorable. We omit the details as the argument is entirely similar to that of Lemma $\ref{lem:cycle-Hall-not-colorable}$. 

We now resolve the remaining cases. Throughout, let $V(F)=\{v_1,v_2,v_3,v_4\}$ and $
V(H)\setminus V(F)=\{a,b,c,d\}$.
For each of the cases below, we construct a graph $G$ in Figure~\ref{fig:remaining-4plus4-cases} such that $(H,G)$ satisfies Hall's condition, but $H$ is not $G$-colorable. Note that, when verifying Hall's condition, it suffices in each of the following cases to consider the densest graph $H$ satisfying the hypothesis of the particular case. Then Hall's condition follows for all spanning subgraphs of $H$ by Proposition~\ref{prop:basics-Hall}. On the other hand, when showing non-colorability, it suffices to consider the sparsest graph $H$, due to Proposition~\ref{prop:basics}. We present these routine verifications in Appendix~\ref{app:prop-9}.

\noindent\textit{Case I(a).} $F\cong K_4-e$ and every vertex
in $\{a,b,c,d\}$ is adjacent to both degree-$3$ vertices of $F$.
Label $F$ so that $v_1,v_3$ are its degree-$3$ vertices and
$v_2v_4\notin E(H)$. If the graph $\overline H-\{v_1,v_3\}$ contains an induced $K_2 \sqcup K_2$, then the corresponding four
vertices induce a $C_4$ in $H$, and we are in Case II. If it contains
an induced path $x_1x_2x_3x_4$, then $H[\{v_1,x_1,x_2,x_4\}]\cong K_4-e,$
whose degree-$3$ vertices are $v_1$ and $x_4$, while
$x_3x_4\notin E(H)$, and we are in Case I(b).

We may therefore suppose that
$\overline H-\{v_1,v_3\}$ contains neither an induced $K_2 \sqcup K_2$ nor an 
induced path $P_4$. If it were also triangle-free, then its unique
nontrivial component would be complete bipartite, and consequently
$H\cong K_8-K_{s,t}$, contrary to our assumption. Hence
$\overline H-\{v_1,v_3\}$ contains a triangle. If the edge $v_2v_4$ lies in some triangle of $\overline H-\{v_1,v_3\}$, then, relabeling its third vertex as $d$, we may assume that $\{v_2,v_4,d\}$ is an independent set in $H$. In this case, we use the graph $G$ in the left column of Figure~\ref{fig:remaining-4plus4-cases}. Suppose now that $v_2v_4$ lies in no triangle of $\overline H-\{v_1,v_3\}$, and let $T$ be any triangle in $\overline H-\{v_1,v_3\}$.  Then at least one of $v_2,v_4$, say $v_4$, has at most one neighbor in $T$ (in $\overline{H}$). Choose two vertices, say $x,y$, in $T$ such that $v_4x, v_4y \in E(H)$. Then $H[\{v_1,v_4,x,y\}]\cong K_4-e,$
with degree-$3$ vertices $v_1$ and $v_4$. Since $v_2v_4\notin E(H)$, we are again in Case~I(b). 

\noindent\textit{Case I(b).} $F\cong K_4-e$ and some vertex
$d\in\{a,b,c,d\}$ is not adjacent to a degree-$3$ vertex of $F$.
Label $F$ so that $v_2v_4, dv_3\notin E(H)$,
where $v_1,v_3$ are the degree-$3$ vertices of $F$. In this case, we use the graph $G$ in the right column of
Figure~\ref{fig:remaining-4plus4-cases}.

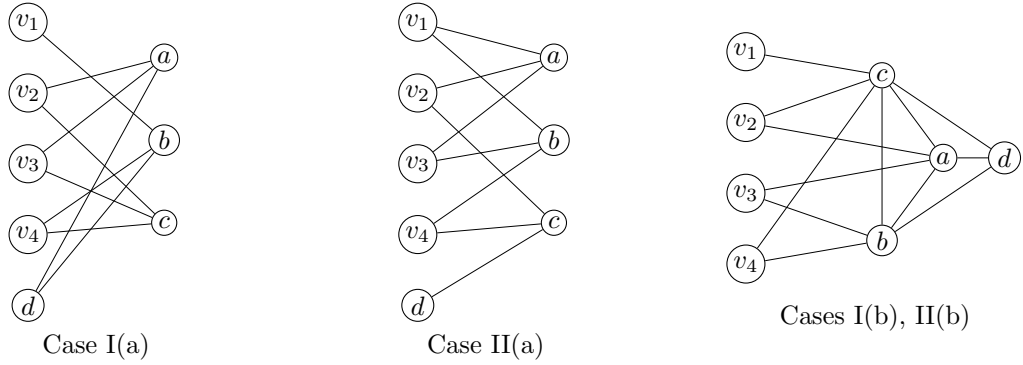
\begin{figure}[h]
\centering

\begin{minipage}{0.32\linewidth}
\centering
\begin{tikzpicture}[
  x=0.9cm, y=0.78cm,
  v/.style={circle,draw,inner sep=1.2pt},
  col/.style={circle,draw,inner sep=1.2pt},
  edge/.style={line width=0.35pt},
  every node/.style={font=\small}
]

\node[v] (v1) at (0,2.8) {$v_1$};
\node[v] (v2) at (0,1.6) {$v_2$};
\node[v] (v3) at (0,0.4) {$v_3$};
\node[v] (v4) at (0,-0.8) {$v_4$};
\node[v] (d)  at (0,-2.0) {$d$};

\node[col] (a) at (2.0,2.2) {$a$};
\node[col] (b) at (2.0,0.8) {$b$};
\node[col] (c) at (2.0,-0.6) {$c$};

\draw[edge] (v1) -- (b);

\draw[edge] (v2) -- (a);
\draw[edge] (v2) -- (c);

\draw[edge] (v3) -- (a);
\draw[edge] (v3) -- (c);

\draw[edge] (v4) -- (b);
\draw[edge] (v4) -- (c);

\draw[edge] (d) -- (a);
\draw[edge] (d) -- (b);

\node[draw=none] at (1.0,-2.7) {Case I(a)};

\end{tikzpicture}
\end{minipage}
\hfill
\begin{minipage}{0.32\linewidth}
\centering
\begin{tikzpicture}[
  x=0.9cm, y=0.78cm,
  v/.style={circle,draw,inner sep=1.2pt},
  col/.style={circle,draw,inner sep=1.2pt},
  edge/.style={line width=0.35pt},
  every node/.style={font=\small}
]

\node[v] (v1) at (0,2.8) {$v_1$};
\node[v] (v2) at (0,1.6) {$v_2$};
\node[v] (v3) at (0,0.4) {$v_3$};
\node[v] (v4) at (0,-0.8) {$v_4$};
\node[v] (d)  at (0,-2.0) {$d$};

\node[col] (a) at (2.0,2.2) {$a$};
\node[col] (b) at (2.0,0.8) {$b$};
\node[col] (c) at (2.0,-0.6) {$c$};

\draw[edge] (v1) -- (a);
\draw[edge] (v1) -- (b);

\draw[edge] (v2) -- (a);
\draw[edge] (v2) -- (c);

\draw[edge] (v3) -- (a);
\draw[edge] (v3) -- (b);

\draw[edge] (v4) -- (b);
\draw[edge] (v4) -- (c);

\draw[edge] (d) -- (c);

\node[draw=none] at (1.0,-2.7) {Case II(a)};

\end{tikzpicture}
\end{minipage}
\hfill
\begin{minipage}{0.32\linewidth}
\centering
\begin{tikzpicture}[
  x=0.9cm, y=0.78cm,
  v/.style={circle,draw,inner sep=1.2pt},
  col/.style={circle,draw,inner sep=1.2pt},
  edge/.style={line width=0.35pt},
  every node/.style={font=\small}
]

\node[v] (v1) at (0,2.8) {$v_1$};
\node[v] (v2) at (0,1.6) {$v_2$};
\node[v] (v3) at (0,0.4) {$v_3$};
\node[v] (v4) at (0,-0.8) {$v_4$};

\node[col] (c) at (2.0,2.4) {$c$};
\node[col] (b) at (2.0,-0.4) {$b$};
\node[col] (d) at (3.8,1.0) {$d$};
\node[col] (a) at (2.9,1.0) {$a$};

\draw[edge] (v1) -- (c);

\draw[edge] (v2) -- (a);
\draw[edge] (v2) -- (c);

\draw[edge] (v3) -- (a);
\draw[edge] (v3) -- (b);

\draw[edge] (v4) -- (b);
\draw[edge] (v4) -- (c);

\draw[edge] (c) -- (b);
\draw[edge] (c) -- (d);
\draw[edge] (b) -- (d);
\draw[edge] (a) -- (c);
\draw[edge] (a) -- (b);
\draw[edge] (a) -- (d);

\node[draw=none,align=center] at (1.9,-1.7) {Cases I(b), II(b)};

\end{tikzpicture}
\end{minipage}

\caption{Constructions of $G$ for the remaining $|V(F)|= |V(H) \setminus V(F)|=4$ cases.}
\label{fig:remaining-4plus4-cases}
\end{figure}

\noindent\textit{Case II(a).} $F\cong C_4$ and every vertex in
$\{a,b,c,d\}$ is adjacent to every vertex of $F$.
We use the construction in the middle column of
Figure~\ref{fig:remaining-4plus4-cases}.

\noindent\textit{Case II(b).} $F\cong C_4$ and some vertex
$d\in\{a,b,c,d\}$ is not adjacent to some vertex of $F$.
Relabeling the cycle, we may assume that $dv_3\notin E(H).$
Since $v_2$ and $v_4$ are the two neighbors of $v_3$ on $C_4$, we also have $v_2v_4\notin E(H)$. Thus the same construction as in Case I(b), shown in the right column
of Figure~\ref{fig:remaining-4plus4-cases}, applies.

In each case, the construction of $G$ in
Figure~\ref{fig:remaining-4plus4-cases} is such that $(H,G)$ satisfies Hall's condition,
but $H$ is not $G$-colorable. As such, $H$ is not $\mathcal G$-Hall-colorable. This completes the proof.
\end{proof}

\subsection{Problem~\ref{prob:A} characterization}
Combining results of Theorem~\ref{thm:G-Hall-characterization} and Theorem~\ref{thm:not-Hall}, along with the exceptional cases in Proposition~\ref{prop:K8-minus-Kst} and Proposition~\ref{prop:four-extra-characterization}, we obtain the following complete answer to Problem~\ref{prob:A}.

\begin{cor}
A graph $H$ is $\mathcal G$-Hall-colorable if and only if either $|V(H)\setminus V(F)|\leq 2$
for every induced subgraph $F\subseteq H$ with $F\in\mathcal F$, or $H\cong K_8-K_{s,t}$
for some $s,t\geq1$ with $s+t\leq6$.
\end{cor}

\section{Hall-universal Graphs}\label{sec:prob2}
The current section focuses on Problem~\ref{prob:B}. We refer to the graphs satisfying the property in Problem~\ref{prob:B} as {Hall-universal graphs}, as defined below.

\begin{defn}\label{def:Hall-universal}
A graph $G\in\mathcal G(V)$ is called \textit{Hall-universal} if, for every $H\in\mathcal G(V)$, whenever $(H,G)$ satisfies Hall's condition, $H$ is $G$-colorable.
\end{defn}

\begin{notation}\label{note:sets}
Let $G \in \mathcal{G}(V)$, and define
\[
\HC(G) = \{H\in\mathcal G(V):(H,G)\text{ satisfies Hall's condition}\},
\]
and
\[
\C(G) = \{H\in\mathcal G(V):H\text{ is }G\text{-colorable}\}.
\]
Let $\MHC(G)$ and $\MC(G)$ denote the sets of edge-maximal members of $\HC(G)$ and $\C(G)$, respectively. For instance, $H\in\MHC(G)$ if $H\in\HC(G)$ and $H+e\notin\HC(G)$ for every $e\in E(\overline H)$. The set $\MC(G)$ is defined analogously.
\end{notation}

With this notation, a graph $G$ is Hall-universal if and only if $\HC(G) = \C(G)$. Note that the inclusion $\C(G)\subseteq 
\HC(G)$ always holds by Lemma~\ref{lem:HJ-lem}. It turns out the notion of maximality characterizes Hall-universal graphs.

\begin{prop}\label{prop:maximal-minimal}
A graph $G$ is Hall-universal if and only if 
 $\MC(G)=\MHC(G)$.
\end{prop}

\begin{proof}
First, suppose that $G$ is Hall-universal. Then we have $\HC(G)=\C(G)$. Therefore, their edge-maximal  elements must be equal.

Next, suppose that
$\MC(G)=\MHC(G)$. We will show that
$\HC(G)\subseteq \C(G)$. Let $H\in \HC(G)$. By adding edges if needed, we may extend $H$ to a maximal graph $H'\in \MHC(G)$. By hypothesis, $H'\in \MHC(G)=\MC(G),$
so $H'$ is $G$-colorable. Since $H\subseteq H'$, $H$ is $G$-colorable by Proposition~\ref{prop:basics}. Therefore, $H\in \C(G)$. 
\end{proof}

\begin{rmk}
 Let $\mNHC(G)$ and $\mNC(G)$ denote the edge-minimal members of $\mathcal G(V)\setminus\HC(G)$ and $\mathcal G(V)\setminus\C(G)$, respectively. For example, $H\in\mNHC(G)$ if $H\notin\HC(G)$ and $H-e\in\HC(G)$ for every $e\in E(H)$. The set $\mNC(G)$ is defined analogously. Using similar reasoning as in Proposition~\ref{prop:maximal-minimal}, we can show that a graph $G$ is Hall-universal if and only if 
 $\mNC(G)=\mNHC(G)$.
\end{rmk}

\begin{prop}\label{prop:star-Hall}
Let $G \cong K_{1,n-1}$ be a star graph, where $n \geq 2$. Then $G$ is Hall-universal.
\end{prop}
\begin{proof}
We will show $ \MC(G) = \{G\} =\MHC(G)$. It follows that stars are Hall-universal by Proposition~\ref{prop:maximal-minimal}.

Suppose a graph $H$ is properly colored by the star $K_{1,n-1}$ centered at vertex $v_n$ with leaves $v_1,\ldots,v_{n-1}$. The list of colors is such that $L(v_n) = \{v_1, \ldots, v_{n-1}\}$ and $L(v_1)=\cdots =L(v_{n-1})=\{v_n\}.$ Then the vertices $\{v_1, \ldots, v_{n-1}\}$ cannot be adjacent to each other. By Theorem~\ref{thm:self-coloring}, the graph $G$ is $G$-colorable. Hence, we obtain $\MC(G)=\{G\}$.

On the other hand, by Corollary~\ref{cor:self-coloring-cor}, the pair $(G,G)$ satisfies Hall's condition. If any two of $v_1,\ldots,v_{n-1}$ are adjacent, then the resulting edge does not satisfy Hall's condition, as the independence number of that edge is only 1. Hence, we obtain $\MHC(G)=\{G\}$.
\end{proof}

\begin{rmk}
In general, there is no known systematic procedure to describe members of $\MHC(G)$ or $\MC(G)$ for an arbitrary graph $G$. Ariel Cook's PhD dissertation \cite{Cook2025ExtremalProblemsGraphReferential} with Peter Johnson contains results on $\MC(G)$ and $\mNC(G)$ for special classes of graphs $G$.
\end{rmk} 

\subsection{Totally Hall-universal graphs} We say $G \in \mathcal{G}(V)$ is \textit{totally Hall-universal} if $\MHC(G) = \{K_V\}$, where $K_V$ is the complete graph on vertex set $V$. By Theorem \ref{thm:HJ-thm}, this immediately implies $\MC(G)=\{K_V\}$. This family of graphs is precisely the one characterized in Theorem~\ref{thm:hedet-main}, which we may restate as follows:

\begin{thm} [\cite{hedet2018, self-coloring}] \label{thm:hedet}
A graph $G$ is totally Hall-universal if and only if it contains a spanning $\{1,2\}$-factor.
\end{thm}

\begin{rmk}\label{rmk:tree-coloring}
If $G$ is bipartite, then by Theorem~\ref{thm:hedet}, $G$ is totally Hall-universal if and only if $G$ has a perfect matching. This implies, for instance, that hypercubes and grid graphs of even order are totally Hall-universal. Many graphs known to contain Hamiltonian cycles \cite{Gould2014RecentAdvancesHamiltonianIII} are also totally Hall-universal, again by Theorem~\ref{thm:hedet}. 
\end{rmk}

The following result establishes that if one is to randomly pick a graph from a pile of all graphs, it almost surely is a Hall-universal graph.

\begin{thm}\label{thm:almost-all}
Almost all graphs are totally Hall-universal.
\end{thm}

\begin{proof}
A uniformly random labeled graph on $n$ vertices has distribution
$G(n,\tfrac12)$, where $G(n,p)$ denotes the well-known binomial random graph model (see, e.g., \cite{frieze2016}). Thus it suffices to show that
$G\sim G(n,\tfrac12)$ is asymptotically almost surely (a.a.s.) totally Hall-universal. It is a classical result of Pósa~\cite{posa1976} that there exists a constant $C>0$
such that if $p\geq C\frac{\log n}{n},
$ then $G(n,p)$ is a.a.s. Hamiltonian. In particular,
$G(n,\tfrac12)$ is a.a.s. Hamiltonian. By Theorem~\ref{thm:hedet}, we conclude $G$ is a.a.s. totally
Hall-universal.
\end{proof}

\subsection{Constructing Hall-universal graphs}
\begin{notation}
    For $k \in \mathbb{N},$ we define $[k]  = \{1,2,\ldots,k\}$.
\end{notation}

\begin{prop}\label{prop:internal-perfect-graphs}
Let $G_0$ be a totally Hall-universal graph, and let $S_1,\ldots,S_k$ be pairwise vertex-disjoint stars, each on at least two vertices and disjoint from $G_0$. Construct $G$ from the disjoint union $G_0\sqcup S_1\sqcup\cdots\sqcup S_k$
by adding arbitrary edges among the centers of the stars and between the centers and vertices of $G_0$. Then $G$ is Hall-universal.
\end{prop}
\begin{proof}
For each $i\in[k]$, let $V_i$ be the set of leaves of $S_i$, and let $V = V(G)$. Define
$H^* = K_V-(K_{V_1}\cup\cdots\cup K_{V_k})$.
We claim that
\[
\MC(G)=\{H^*\}=\MHC(G),
\]
and the result follows by Proposition~\ref{prop:maximal-minimal}.

Since $G_0$ is totally Hall-universal, $K_{V(G_0)}$ is $G_0$-colorable.
Fix such a coloring. For each star $S_i$ with center $u_i$, choose a leaf
$v_i\in V_i$ and assign
\[
\varphi(u_i)=v_i
\qquad\text{and}\qquad
\varphi(w)=u_i \quad\text{for all }w\in V_i.
\]
This gives a proper
$G$-coloring of $H^*$. 

Now suppose that $(H,G)$ satisfies Hall's condition. If $H$ contains an
edge $xy$ with $x,y\in V_i$ for some $i$, then $N_G(x)=N_G(y)=\{u_i\},$ and the edge $xy$ does not satisfy Hall's condition. Hence no such edge exists,
and we have $H\subseteq H^*$. It follows that $\MC(G)=\{H^*\}=\MHC(G).$
\end{proof}

\begin{rmk} \label{rmk:caterpillars}
Proposition~\ref{prop:internal-perfect-graphs} applies to
trees such that a removal of the pendant stars along with their parent leaves a forest with a perfect matching.
In particular, Proposition~\ref{prop:internal-perfect-graphs} applies to
caterpillars in which every non-leaf vertex has at least one leaf neighbor.
Indeed, every spine vertex is then the center of a nontrivial star, so removing
stars removes the entire spine and leaves the empty forest, which trivially has a
perfect matching.
\end{rmk}

The consequence noted in Remark~\ref{rmk:caterpillars} 
asks for a characterization of trees that are Hall-universal. We show next that all forests are Hall-universal.

\subsection{Forests} 
\begin{thm}\label{thm:trees-Hall-universal}
Every forest is Hall-universal.
\end{thm}

\begin{proof}
We prove the following stronger statement. Let $F$ be a forest on vertex set
$V$. For each $v\in V$, we say the set of colors $P_v$ is \textit{private} if
\begin{equation} \label{eqn:private}
    P_v\cap V=\varnothing \quad \text{for all } v\in V,
\qquad
P_v\cap P_u=\varnothing \quad \text{for all } u\neq v.
\end{equation}
Define
\[
L(v)=N_F(v)\cup P_v \quad \text{ for all } v\in V.
\]
For a graph $H \in \mathcal{G}(V)$, we show that if $(H,L)$ satisfies Hall's condition, then $H$ is $L$-colorable. If $F$ is a forest on vertex set $V$, we take $P_v=\varnothing$ for all $v\in V$. Then $L(v)=N_F(v)$, and the
result follows.

We argue by induction on $|V|$. Assume $|V|\geq 1$ and that the claim holds for all forests on smaller vertex
sets. Consider a graph $H$ on $V$ such that $(H,L)$ satisfies Hall's condition.

First, suppose $F$ has an isolated vertex $x$. Then $L(x)=P_x$. We must have $P_x\neq\varnothing$, since $(H,L)$ satisfies Hall's condition. Let $V'=V\setminus\{x\}$ and let $L'(v)=L(v)$ for all $v\in V'$. 
Since $x$ is isolated in $F$, we have
\[
L'(v)=N_{F-\{x\}}(v)\cup P_v
\quad\text{for all }v\in V'.
\]
Hence $(H[V'],L')$ satisfies Hall's condition. By the induction hypothesis $H[V']$ is $L'$-colorable.  Such a coloring extends to a proper coloring of $H$ by assigning $x$ any color from $P_x$.

Next, suppose $F$ has a leaf $x$ with $P_x\neq\varnothing$. Let $u$ be the
unique neighbor of $x$ in $F$. Define
\[
V'=V\setminus\{x\}, \qquad F'=F[V'],
\]
and define private sets on $V'$ by
\[
P'_v=
\begin{cases}
P_v\cup\{x\}, & \text{if } v=u,\\
P_v, & \text{otherwise},
\end{cases}
\qquad\text{for all } v\in V'.
\]
Note that these sets satisfy the assumptions in \eqref{eqn:private}. Moreover,
\[
L'(v)=N_{F'}(v)\cup P'_v=N_F(v)\cup P_v = L(v)
\quad\text{for all } v\in V'.
\]
Hence the pair $(H[V'],L')$ satisfies Hall's condition. By the induction hypothesis, $H[V']$ is $L'$-colorable, and this
coloring extends to $H$ by assigning $x$ any color in $P_x$.

We may therefore assume that $F$ has no isolated vertices and that every leaf
$y$ of $F$ satisfies $P_y=\varnothing$. Choose a longest path in $F$. Let $x$
be an endpoint of this path, and let $u$ be the neighbor of $x$ on the path.
Set
\[
A=\{\,y\in V:y\text{ is a leaf of }F\text{ and }yu\in E(F)\,\}.
\]
In other words, $A$ is the set of sibling leaves containing vertex $x$. Then $A\neq\varnothing$, and $u$ has at most one neighbor in $F$ outside $A$.
If such a neighbor exists, denote it by $w$. Note that we have $L(y)=\{u\}$ for all $y\in A$ and
the set $A$ is independent in $H$. Indeed, if two vertices of $A$ were
adjacent in $H$, then Hall's condition would fail on the induced edge between
them. Now consider
\[
V'=V\setminus(A\cup\{u\}),
\qquad
F'=F[V'].
\]
Define private sets on $V'$
by
\[
P'_v=
\begin{cases}
P_v\cup\{u\}, & \text{if $w$ exists, $v=w$, and $N_H(w)\cap A=\varnothing$,}\\
P_v, & \text{otherwise}.
\end{cases}
\qquad\text{for all } v\in V'.
\]
Again these sets satisfy the assumptions in \eqref{eqn:private} and
\[
L'(v)=N_{F'}(v)\cup P'_v=N_F(v)\cup P_v = L(v)
\quad\text{for all } v\in V'\setminus \{w\}.
\]
If $w$ does not exist or $N_H(w)\cap A=\varnothing$, then
$L'(v)=L(v)$ for all $v\in V'$, so $(H[V'],L')$ satisfies Hall's
condition. In this case, choose any $z\in A$. Suppose now that $w$ exists and $N_H(w)\cap A\neq\varnothing$. 
Choose a vertex $z\in N_H(w)\cap A$. As $(H,L)$ satisfies Hall's
condition, for every $U\subseteq V'$ we have
\[
|U|+1
\leq
\sum_{\sigma}\alpha\big(H[U\cup\{z\}],L,\sigma\big).
\]
Among the vertices of $U\cup\{z\}$, the only vertices whose lists contain
$u$ are $z$ and $w$. As these two vertices are adjacent in $H$, we have $\alpha(H[U\cup\{z\}],L,u)=1$.
For every color $\sigma\neq u$, the vertex $z$ does not contribute to
$\alpha(H[U\cup\{z\}],L,\sigma)$, and the change from $L$ to $L'$ only affects
the color $u$ at vertex $w$. This means
\[
\alpha(H[U\cup\{z\}],L,\sigma)
=
\alpha(H[U],L',\sigma)
\quad\text{for all }\sigma\neq u.
\]
Therefore,
\[
|U|+1
\leq
1+\sum_{\sigma}\alpha(H[U],L',\sigma),
\]
and hence $(H[V'],L')$ satisfies Hall's condition in both cases.

By the induction hypothesis, $H[V']$ has an $L'$-coloring. Extend it by assigning every
vertex of $A$ the color $u$, and by assigning $u$ the color $z$. This is a
valid extension, since $L(y)=\{u\}$ for every $y\in A$, and
$z\in A\subseteq N_F(u)\subseteq L(u)$. The vertices in $A$ may all receive
the color $u$ because $A$ is independent in $H$. No vertex of $V'$ can receive
the color $z$, since the only vertex whose list contains $z$ is $u$. Finally, if $w$ does not exist or $N_H(w)\cap A\neq\varnothing$,
then no vertex of $V'$ can receive the color $u$. If $w$ exists and
$N_H(w)\cap A=\varnothing$, then $w$ may receive the color $u$, but
this creates no conflict because $w$ has no neighbor in $A$ (in $H$). Thus the coloring extends to an $L$-coloring of $H$.

This completes the induction. We conclude every forest is Hall-universal.
\end{proof}

\begin{rmk}
Theorem~\ref{thm:trees-Hall-universal} could equivalently be stated in terms of trees. Indeed, a disjoint union of Hall-universal graphs is
Hall-universal, as Hall's condition restricts to each component of $G$.
\end{rmk}

\subsection{Complete multipartite graphs}
In this subsection, we completely characterize the complete multipartite graphs that are Hall-universal.
\begin{defn} [\cite{CROPPER20061988}]
Given a graph $G$ on vertex set $V$, its \textit{Hall ratio} is defined as
\[
\rho(G)
=
\max_{\varnothing\neq U\subseteq V}
\frac{|U|}{\alpha(G[U])}.
\]
\end{defn}
\begin{prop}\label{prop:multipartite-Hall-ratio}
Let $G\cong K_{A_1,\ldots,A_r}$ be a complete multipartite graph on $V=A_1\sqcup\cdots\sqcup A_r,$
where $r\geq 2$. For all $H \in \mathcal{G}(V)$, the pair $(H,G)$ satisfies Hall's condition if and only if
\[
\rho(H[A_i])\leq |V|-|A_i|
\qquad\text{for all } i\in[r].
\]
\end{prop}

\begin{proof}
Let $n_i=|A_i|$ and $n=|V|$. For every vertex $v\in A_i$, we have $N_G(v)=V\setminus A_i$.
Thus, for every $U\subseteq V$,
\[
\sum_{v\in V}\alpha\bigl(H[U\cap N_G(v)]\bigr)
=
\sum_{i=1}^r n_i \,\alpha\bigl(H[U\setminus A_i]\bigr).
\]
We first show that it is enough to verify Hall's condition for subsets $U \subseteq A_i$, $i\in [r]$, i.e., subsets contained within a single partition. If $U$ intersects at least two parts of the multipartition, then
$U\setminus A_i$ is nonempty for all $i\in [r]$, and  $\alpha\bigl(H[U\setminus A_i]\bigr)\geq 1$ for all $i \in [r]$. Therefore,
\[
\sum_{i=1}^r n_i \,\alpha\bigl(H[U\setminus A_i]\bigr)
\geq
\sum_{i=1}^r n_i
=
n
\geq |U|,
\]
meaning Hall's condition holds for such $U$.

Now suppose $U\subseteq A_j$ for some $j\in[r]$. In this case, $U\setminus A_j=\varnothing$, and $U\setminus A_i=U$ for $i\neq j$. Thus checking Hall's condition is equivalent to checking
\[
|U|
\leq
\sum_{i\neq j} n_i \,\alpha(H[U])
=
(n-n_j) \,\alpha(H[U]).
\]
Said otherwise, $(H,G)$ satisfies Hall's condition if and only if $|U|\leq (n-|A_j|) \; \alpha(H[U])
$
for every $j$ and every $U\subseteq A_j$. This is precisely the condition $\rho(H[A_j])\leq n-|A_j|$ for all $j\in [r]$,
as claimed.
\end{proof}

\begin{lem}\label{lem:k-colorable}
 Let $k\geq 2$. If $G$ is a graph on at most $k+2$ vertices such that $\rho(G) \leq k$, then $G$ is $k$-colorable.
\end{lem}
\begin{proof}
The claim is immediate if $|V(G)|\leq k$. If $|V(G)|=k+1$, then $G$ is not
$k$-colorable only when $G\cong K_{k+1}$. However, in this case, we have $\rho(G) = k+1$, which contradicts the assumption $\rho(G) \leq k$.

Now suppose $|V(G)|=k+2$. The graph $G$ contains no
clique of order $k+1$, owing to the previous case. Equivalently, for every vertex $v$, the graph
$\overline G-v$ contains an edge. If $\overline G$ contains two
disjoint edges, then $G$ has two disjoint nonadjacent pairs. We may assign one color to each nonadjacent pair and assign remaining vertices distinct colors, which uses $k$ colors.
If $\overline G$ contains a triangle, then we can assign one color to the independent triangle in $G$
and assign remaining vertices distinct colors, which also uses $k$ colors. In both cases, $G$ is $k$-colorable.

It remains only to observe that one of these two cases must occur. Indeed, if $\overline G$ contains neither two disjoint edges nor a triangle, then all edges of $\overline G$ are incident with a common vertex $v$. Hence $\overline G-v$ is edgeless, contradicting the fact that $\overline G-v$ contains an edge for every vertex $v$.
\end{proof}

\begin{thm}\label{thm:complete-multipartite-Hall-universal}
Let $G\cong K_{A_1,\ldots,A_r}
$ be a complete multipartite graph on $V=A_1\sqcup\cdots\sqcup A_r$, with
$r\geq 2$. 
Then $G$ is Hall-universal if and only if
\[
|V|-\max_i |A_i| \leq 2
\qquad\text{or}\qquad
2\max_i |A_i|  - |V| \leq 2.
\]
\end{thm}
\begin{proof}
Let $n=|V|$, $M=\max_i |A_i|$, and $k = n-M$. We first establish the ``if'' direction. Suppose first that $2M\leq n$. Equivalently, no part of the multipartition has
more than half of the vertices. In this case, $G$ contains a spanning $\{1,2\}$-factor, by Fact~\ref{fact:12-factor}, and $G$ is totally Hall-universal, by Theorem~\ref{thm:hedet}. Now suppose $2M>n$. Then there is a unique largest part, say $A_0$, so that $|A_0|=M$. Let $H$ be a graph on $V$ such that $(H,G)$ satisfies Hall's condition. By
Proposition~\ref{prop:multipartite-Hall-ratio}, we have $\rho(H[A_0])\leq k$.
We claim $H[A_0]$ is $k$-colorable. If $k=n-M=1$, then $H[A_0]$ must be edgeless and hence $1$-colorable. If $k=2$, then Hall's condition is equivalent to
\begin{equation} \label{eq:k=2}
    |U|\leq 2 \,\alpha(H[U]) \quad \text{ for all $U\subseteq A_0$.
}
\end{equation}
If $H[A_0]$ contained an odd cycle, then there would exist $U$ such that $H[U]$ is a shortest induced odd cycle. If $|U|=2\ell +1$, then $\alpha(H[U])=\ell$, which
contradicts inequality \eqref{eq:k=2}. Therefore $H[A_0]$ is bipartite and
hence $2$-colorable. Finally, if $k \geq 3$ and $M\leq k+2$, we obtain that $H[A_0]$ is $k$-colorable by Lemma~\ref{lem:k-colorable}.

We now color $H$. Since every vertex of $A_0$ has list $V\setminus A_0$ of size $k$, we can properly color $H[A_0]$. Next, we can color the vertices of $V\setminus A_0$ injectively
using distinct colors from $A_0$. This is possible because  $M=|A_0| > |V\setminus A_0| =n-M$. Therefore, $H$ is $G$-colorable whenever $(H,G)$ satisfies Hall's condition.

We now prove the ``only if'' direction by contrapositive. Suppose 
\[
k=n-M\geq 3
\qquad\text{and}\qquad
M\geq k+3.
\]
In this case, we construct a graph $H$ such that $(H,G)$ satisfies Hall's condition but $H$ is not $G$-colorable.
Let $A_0$ be a largest partition of $G$ such that $|A_0|=M$. Choose a subset $W\subseteq A_0$ with $|W|=k+3$. Construct the join
\[
H[W]=K_{k-2}\vee C_5,
\]
i.e., connect each vertex of the $K_{k-2}$ to every vertex of the $C_5$. Let all remaining vertices
of $V$ be isolated, and call the resulting graph $H$. 

We first show that $(H,G)$ satisfies Hall's condition. By
Proposition~\ref{prop:multipartite-Hall-ratio}, it suffices to check
\[
\rho(H[A_i])\leq |V|-|A_i| \quad \text{for all } i \in [r].
\]
If $i\neq 0$, then $H[A_i]$ is edgeless, so $\rho(H[A_i])=1\leq |V|-|A_i|$. We claim that every nonempty subset $X\subseteq W$ satisfies
\[
|X|\leq k \, \alpha(H[X]),
\]
which means $\rho(H[W])\leq k=|V|-|A_0|$.  This implies $\rho(H[A_0])\leq |V|-|A_0|$ as additional isolated vertices cannot increase the Hall ratio. We now prove the stated claim.

Let $s$ be the number of vertices of the clique $K_{k-2}$ and $t$
be the number of vertices of the cycle $C_5$ in $X \subseteq W$. Thus
\[
|X|=s+t,
\qquad
0\leq s\leq k-2,
\qquad
0\leq t\leq 5.
\]
If $t=0$, then $H[X]$ is a clique of order $s\leq k-2$, so $\alpha(H[X]) = 1$. Hence, we have
\[
|X|=s\leq k=k \; \alpha(H[X]).
\]
If $1\leq t\leq 2$, then $\alpha(H[X])\geq 1$, and
\[
|X|=s+t\leq (k-2)+2=k\leq k \, \alpha(H[X]).
\]
If $t\geq 3$, then $C_5[X]$ contains an
independent set of size at least $2$. Hence $\alpha(H[X])\geq 2$, and
\[
|X|=s+t\leq (k-2)+5=k+3\leq 2k\leq k \, \alpha(H[X]),
\]
since $k\geq 3$. This proves our claim. Therefore $(H,G)$ satisfies Hall's condition.

It remains to show that $H$ is not $G$-colorable. Every vertex of $W \subseteq A_0$ has
available colors exactly $V\setminus A_0$, where $|V\setminus A_0| = k$. But by construction, the chromatic number of $K_{k-2}\vee C_5$ is $(k-2)+3=k+1$. Thus the induced subgraph $H[W]$ is not $k$-colorable. Hence $H$ is not $G$-colorable.
\end{proof}

\begin{cor}\label{thm:Kab-characterization}
Let $G\cong K_{A,B}$ be a complete bipartite graph. Then $G$ is Hall-universal if and only if
\[
\min\{|A|,|B|\}\le 2
\qquad\text{or}\qquad
\bigl||A|-|B|\bigr|\le 2.
\]
\end{cor}

\subsection{Grid graphs} 

The grid graphs $G=P_m\square P_n$ with $mn$ even have a perfect matching, meaning they are Hall-universal, by Theorem~\ref{thm:hedet}. We show in this subsection that such is also the case when they do not have a perfect matching. We first establish the following elementary lemma.

\begin{lem}\label{lem:odd-grid-tromino-matching} Let $G=P_m\square P_n$ with $m,n$ odd. Consider two distinct vertices $u,w$ in the larger bipartition class of $G$, and suppose that $u$ and $w$ have a common neighbor $v$. Then $G-\{u,v,w\}$ has a perfect matching.
\end{lem}

\begin{proof}
The result is trivial for $mn=1$, so assume $mn\geq 3$.
Identify $G$ with the Cartesian grid $[m]\times[n]$.
We induct on $m+n$. If either the first two or the last two rows are disjoint from
$\{u,v,w\}$, we can perfectly match the vertices in those two rows and apply the
induction hypothesis to the remaining $(m-2)\times n$ grid.  The same applies to the first two or the last two columns.

Since $\{u,v,w\}$ occupies at most three consecutive rows or columns, this reduction is always possible unless,
up to symmetry, the grid has dimensions
\[
3\times 1,\qquad 3\times 3,\qquad 3\times 5.
\]
In each of these cases, it is easy to check $G-\{u,v,w\}$ has a perfect matching.
\end{proof}

\begin{thm}\label{thm:odd-grid-Hall-universal}
Let $G=P_m\square P_n$ be a grid graph. Then $G$ is Hall-universal.
\end{thm}

\begin{proof}
    As noted earlier, if either $m$ or $n$ is even, then $G$ has a $1$-factor. In particular, this means $\MC(G) = \{K_V \} = \MHC(G)$, and $G$ is totally Hall-universal by Theorem~\ref{thm:hedet}. So suppose both $m,n$ are odd, and $G$ is not totally Hall-universal. Let $V= V(G)=X\sqcup Y$ be the bipartition of $G$, such that $|X|=|Y|+1$. We show that $\MC(G)=\MHC(G)$, and the result follows by
Proposition~\ref{prop:maximal-minimal}.

\begin{claim}\label{claim:grid}
We have $H \in \MHC(G)$ if and only if $H = H_{uw}  \coloneq K_{V}-uw$ for some $u,w\in X$ such that $N_G(u)\cap N_G(w)\neq\varnothing$.
\end{claim}

Suppose Claim~\ref{claim:grid} holds. Let $u,w\in X$ with $N_G(u)\cap N_G(w)\neq\varnothing$, and choose
$v\in N_G(u)\cap N_G(w)$. By Lemma~\ref{lem:odd-grid-tromino-matching},
$G-\{u,v,w\}$ has a perfect matching $M$. Color both $u$ and $w$ with $v$,
color $v$ with $u$, and for each edge $xy\in M$, color $x$ with $y$ and $y$ with $x$. This is a proper $G$-coloring of $K_{V}-uw$. 
Therefore, every graph in $\MHC(G)$ lies in $\C(G)$. Indeed, since adding any additional edge results in $K_V$, and since we know $K_V \notin \MC(G)$, we conclude every graph in $\MHC(G)$ lies in $\MC(G)$. We now prove Claim~\ref{claim:grid}.

 \noindent\textit{Proof of Claim~\ref{claim:grid}.} Suppose $H\in\MHC(G)$. If $\alpha(H[X],L,\sigma) =1$ for all $\sigma \in Y$, then $N_G(\sigma)$ is a clique in $H[X]$ for all $\sigma\in Y$. Since $G$ is connected, it follows that $H[X]$ is connected. As each vertex of $X$ receives members of $Y$ as colors, we obtain
\[
\sum_{\sigma\in V}\alpha(H[X],L,\sigma)
=\sum_{\sigma\in Y}\alpha(H[X],L,\sigma)
=|Y|<|X|.
\]
But this contradicts Hall's condition. Hence there must exist $y\in Y$ such that $\alpha(H[X],L,y)\geq 2$. In other words, there exist nonadjacent $u,w\in X$ with $N_G(u)\cap N_G(w)\neq\varnothing$.

Conversely, suppose $N_G(u)\cap N_G(w)\neq\varnothing,$ and choose $v\in N_G(u)\cap N_G(w)$. Let $F$ be any connected induced
subgraph of $H_{uw}$, and write
\[
S=V(F)\cap X,\qquad T=V(F)\cap Y.
\]
If $S\neq X$, choose $x_0\in X\setminus S$. Note that $G-x_0$ has a perfect
matching. Indeed, this follows from Lemma~\ref{lem:odd-grid-tromino-matching}: if only one vertex is deleted, then one can add the edge between the other two vertices to the perfect matching given by the lemma. Then by Hall's marriage theorem, we have
\[
|N_G(S)|\ge |S|
\qquad\text{and}\qquad
|N_G(T)|\ge |T|.
\]
Hence, we obtain
\[
\sum_{\sigma\in V}\alpha(F,L,\sigma)
\ge |N_G(S)|+|N_G(T)|
\ge |S|+|T|
=|V(F)|.
\]
If $S=X$, since $u$ and $w$ are nonadjacent in $H_{uw}$ and both
have the color $v$ in their lists, $\alpha(F,L,v)\ge2.$
For every $\sigma\in Y\setminus\{v\}$, we have
$\alpha(F,L,\sigma)\ge1$, and hence
\[
\sum_{\sigma\in Y}\alpha(F,L,\sigma)
\ge |Y|+1=|X|.
\]
Also,
\[
\sum_{\sigma\in X}\alpha(F,L,\sigma)
\geq \sum_{\sigma\in N_G(T)}\alpha(F,L,\sigma) \geq |N_G(T)|\ge |T|.
\]
Therefore, we have
\[
\sum_{\sigma\in V}\alpha(F,L,\sigma)
\ge |X|+|T|
=|V(F)|.
\]
To summarize, we have $H_{uw}\in\HC(G)$. If $H \in \MHC(G)$, as we must have $H \subseteq H_{uw}$, it follows that $H = H_{uw}\in\MHC(G)$. 
\end{proof}

\section{Conclusion and Future Work}\label{sec:conclusion}
Graphic list coloring can be viewed naturally as a constraint satisfaction problem (CSP). The graph $H$ encodes incompatibilities among the objects
to be assigned, while $G$ prescribes the choices at each vertex.  A
$G$-coloring of $H$ is therefore an assignment that avoids the conflicts in
$H$ while respecting the restrictions imposed by $G$. Next, one can also see graphic list coloring from the perspective of list homomorphisms. We hope to return to the study of graphic list coloring from these well-studied perspectives. The following questions are immediate from the current paper. 

\begin{enumerate}

    \item Complete the characterization of Hall-universal graphs.

    We have seen that Hall-universal graphs $G$ must be such that $\MC(G)=\MHC(G)$. Observe that all graphs in $\MC(G)$ must be complete multipartite graphs since maximality forces an edge between vertices of distinct color classes. Hence, a necessary condition for a graph $G$ to be Hall-universal is that members of $\MHC(G)$ must be complete multipartite graphs.
    
    \item {Develop the enumeration theory of graphic list colorings.}
    
    It is of interest to count or bound the number of $G$-colorings of $H$ in terms of known parameters of $G$ and $H$, such as their degree sequence. Can the fact that every $G$-colorable graph $H$ satisfies Hall's condition be used to obtain such bounds?

    \item {Study algorithmic aspects of Hall's condition.}
    
    When the lists are graphic, are there better ways  to check Hall's condition than verifying for each connected induced subgraph?
\end{enumerate}

\section{Declaration of Generative AI and AI-Assisted Technologies in the Writing Process}\label{Disclosure}

During the preparation of this work, the author used ChatGPT-5.6 Sol to resolve exceptional cases in Subsection~\ref{sec:exception}. In particular, it identified counterexamples to the initial expectation that the restriction $|V(H)\setminus V(F)|\neq4$ could be omitted from the statement of Theorem~\ref{thm:not-Hall}, which led the author to formulate Proposition~\ref{prop:K8-minus-Kst}. It also constructed the graphs used for the remaining $8$-vertex cases shown in Figure~\ref{fig:remaining-4plus4-cases}. After using this tool, the author reviewed, edited, and verified the content as needed and takes full responsibility for the integrity and accuracy of the publication.

\section{Acknowledgments}
This work began at the 2025 Virtual Masamu Advanced Study Institute (MASI). MASI is an annual convening of the US-Africa Collaborative Research Network, which is currently supported by the NSF award DMS 2620609. 

The author is deeply grateful to Pete Johnson for introducing the problem and for pointing to the relevant literature.

\printbibliography

\vspace{1em}

\noindent
Parikshit Chalise\\
Department of Applied Mathematics and Statistics, Johns Hopkins University\\
Baltimore, MD 21218, USA\\
E-mail: \href{mailto:pchalis1@jhu.edu}{\textcolor{purple}{\texttt{pchalis1@jhu.edu}}}

\newpage 

\appendix
\section{Details for Proof of Proposition~\ref{prop:four-extra-characterization}}
\label{app:prop-9}

We present here the verifications omitted in the proof of Proposition~\ref{prop:four-extra-characterization} for the four cases associated with
Figure~\ref{fig:remaining-4plus4-cases}. Let $H^+$
denote the densest graph satisfying the hypothesis of each case. Then we have
\[
\begin{array}{c|c}
\text{Case} & H^+\\ \hline
\mathrm{I(a)}
&K_8-\{v_2v_4,v_2d,v_4d\}\\
\mathrm{I(b)}
&K_8-\{v_2v_4,dv_3\}\\
\mathrm{II(a)}
&K_8-\{v_1v_3,v_2v_4\}\\
\mathrm{II(b)}
&K_8-\{v_1v_3,v_2v_4,dv_3\}.
\end{array}
\]
Here Case~I(a) refers to the subcase where $\{v_2,v_4,d\}$ is independent, which is not covered by the remaining cases, as discussed in the proof of Proposition~\ref{prop:four-extra-characterization}.

First, we verify Hall's condition. For each case, let $A=\{v_1,v_2,v_3,v_4,d\}$ and $B=\{a,b,c\}$. We apply Observation~\ref{obs:hall-partition} with the corresponding
partition of the color set $S_A = \{a,b,c\}$, $S_B =\{v_1,v_2,v_3,v_4,d\}$. The graph $H^+[B]$ is colorable from the lists restricted to $S_B$. In
Cases~I(a) and~II(a), one may use
\[
a\mapsto v_3,\qquad b\mapsto v_1,\qquad c\mapsto v_2,
\]
while in Cases~I(b) and~II(b), one may use
\[
a\mapsto v_2,\qquad b\mapsto v_3,\qquad c\mapsto v_1.
\]
It suffices to show the graph $H^+[A]$ satisfies Hall's condition with respect to lists restricted to $S_A$.
For $U\subseteq A$, define
\[
h(U)=\sum_{\sigma\in\{a,b,c\}}
\alpha(H^+[U],L,\sigma).
\]
A direct calculation using the lists in
Figure~\ref{fig:remaining-4plus4-cases} gives
\[
\begin{array}{c|ccccc}
\begin{tikzpicture}[baseline=(current bounding box.center)]
  \draw (2,1) -- (3,-0);
  \node[anchor=north west] at (0.08,0.82)
    {$\displaystyle \min_{\substack{U\subseteq A\\ |U|=r}} h(U)$};
  \node[anchor=south east] at (3.12,0.25) {$r$};
\end{tikzpicture}
& 1 & 2 & 3 & 4 & 5 \\ \hline
\text{Case I(a)}  &1&2&3&4&6\\
\text{Case I(b)}  &1&2&3&4&6\\
\text{Case II(a)} &1&2&3&4&6\\
\text{Case II(b)} &1&2&3&4&6
\end{array}
\]
Observe that we have $h(U)\ge |U|$ for every $U$ in each case. Therefore, $(H^+,G)$ satisfies Hall's condition in all four cases.

It remains to check that $H$ is not $G$-colorable. In Case~I(a), the list assignments are
\[
L(v_1)=\{b\},\qquad L(v_3)=\{a,c\},\qquad
L(v_4)=\{b,c\},\qquad L(d)=\{a,b\}.
\]
Since $v_1$ is adjacent to $v_4$ and $d$, any coloring forces
$v_1\mapsto b$, $v_4\mapsto c$, and $d\mapsto a$. But $v_3$ is adjacent
to both $v_4$ and $d$ and does not have a distinct color available.
In Case~II(a), the list assignments are
\[
L(v_1)=\{a,b\},\qquad L(v_2)=\{a,c\},\qquad
L(v_4)=\{b,c\},\qquad L(d)=\{c\}.
\]
Since $d$ is adjacent to $v_2$ and $v_4$, any coloring forces
$d\mapsto c$, $v_2\mapsto a$, and $v_4\mapsto b$. But $v_1$ is adjacent
to both $v_2$ and $v_4$ and does not have a distinct color available. Finally, in Cases~I(b) and~II(b), the list assignments are
\[
L(v_1)=\{c\},\qquad L(v_2)=\{a,c\},\qquad
L(v_3)=\{a,b\},\qquad L(v_4)=\{b,c\}.
\]
Since $v_1$ is adjacent to $v_2$, and $v_2$ is adjacent to $v_3$, any
coloring forces $v_1\mapsto c$, $v_2\mapsto a$, and $v_3\mapsto b$.
But $v_4$ is adjacent to both $v_1$ and $v_3$ and does not have a
distinct color available. We conclude that $H$ is not $G$-colorable in all four cases.

\end{document}